\documentclass{amsart}
\usepackage{amsmath, amssymb, amsthm, mathrsfs, enumerate}
\newtheorem{theorem}{Theorem}
\newtheorem{proposition}{Proposition}[section]
\newtheorem{lemma}{Lemma}[section]
\newtheorem{claim}{Claim}
\theoremstyle{definition}
\newtheorem{definition}{Definition}[section]

\title[Singularities of the dual curve in positive characteristic]{Singularities of the dual curve of a certain plane curve in positive characteristic}
\author{Kosuke Komeda}

\address{Department of Mathematics, Graduate School of Science, Hiroshima University, 1-3-1 Kagamiyama, Higashi-Hiroshima, 739-8526 Japan}
\email{d181679@hiroshima-u.ac.jp}

\begin{document}
\maketitle

\begin{abstract}
It is well known that the Gauss map for a complex plane curve is birational, whereas the Gauss map in positive characteristic is not always birational. Let $q$ be a power of a prime integer. We study a certain plane curve of degree $q^2+q+1$ for which the Gauss map is inseparable with inseparable degree $q$. As a special case, we show a relation between the dual curve of the Fermat curve of degree $q^2+q+1$ and the Ballico-Hefez curve.
\end{abstract}
\renewcommand{\thefootnote}{\fnsymbol{footnote}}
\renewcommand{\thefootnote}{\arabic{footnote}}
\section{Introduction}
Let $p$ be a prime integer, and $q$ a power of $p$. We work over an algebraically closed field $\Bbbk$ of charcteristic $p$. We consider a plane curve $C$ of degree $q^2+q+1$ defined by a homogeneous polynomial of the form 
\begin{equation}
F=\sum_{i,j,k} a_{ijk} x_i x_j ^q x_k ^{q^2},\label{1}
\end{equation}
where $a_{ijk}$ are coefficients in $\Bbbk$, and $[x_0 : x_1 : x_2 ]$ is a homogeneous coordinate system in $\mathbb{P} ^2 $. If $a_{ijk} $ are general, then the plane curve $C$ is smooth. The condition that the defining polynomial of $C$ is of the form (\ref{1}) is independent of the choice of homogeneous coordinates of $\mathbb{P}^2$ (see Proposition 2.1).

Let $C^{\vee } $ be the dual curve of the plane curve $C$. 
The Gauss map 
\begin{equation}
\Gamma : C\to C^{\vee } ; [x_0 : x_1 : x_2] \mapsto \left[\frac{\partial F}{\partial x_0} : \frac{\partial F}{\partial x_1} : \frac{\partial F}{\partial x_2}\right] \label{8}
\end{equation}
is an inseparable morphism. For every $i$, the partial derivative of $F$ with respect to $x_i$ is
\begin{equation}
\frac{\partial F}{\partial x_i} = \sum_{j, k} a_{ijk} x_j ^q x_k ^{q^2} = \left( \sum_{j,k} \alpha _{ijk} x_j x_{k} ^q \right) ^q,\label{5}
\end{equation}
where $\alpha _{ijk} = a_{ijk} ^{1/q} $. Thus, if $a_{ijk}$ are general, then the inseparable degree of the Gauss map is $q$. The purpose of this paper is to study singularities of the dual curve $C^{\vee}$ of a plane curve $C$ defined by a polynomial of the form (\ref{1}).

We define $\mathscr{C}$ to be the set of all the projective plane curves defined by homogenious polynomials of the form (\ref{1}). Note that $\mathscr{C}$ is identified with $\mathbb{P}^{26}$. 

Note that all tangent lines of the curve $C\in \mathscr{C}$ intersect $C$ with multiplicity at least $q$ at the tangent point. In our case, a double tangent and a flex are defined as following:
\begin{definition}
Let $m$ be an integer at least 2. We define an \emph{m-ple tangent} to be a tangent line of $C$ which has distinct $m$ tangent points with multiplicity $q$, and a \emph{flex} to be a point at which the tangent line intersects $C$ with multiplicity $q+1$. A 2-ple tangent is called a double tangent.
\end{definition}
\begin{theorem}
Suppose that $C$ is a general member of $\mathscr{C}$. Then 
\begin{enumerate}[\normalfont (i)]
\item the degree of the dual curve $C^{\vee}$ is $(q^2+q+1)(q+1)$,
\item the dual curve $C^{\vee}$ has only ordinary nodes as its singularities, 
\item the number of ordinary nodes of $C^{\vee } $ i.e. double tangent lines of C, is 
\[
\frac{q(q^2+q+1)(q^3+3q^2+3q-1)}{2},
\] 
and
\item the number of flexes of C is 
\[
(q^3+2q^2-q+1)(q^2+q+1).
\]
\end{enumerate}
\end{theorem}

We compare our theorem with the classical situation. Let $\tilde{C}$ be a general \textit{complex} plane curve of degree $d$. Then the degree of the dual curve $\tilde{C} ^{\vee } $ is $d(d-1)$. Moreover, each flex of $\tilde{C}$ corresponds to a cusp of $\tilde{C} ^{\vee}$, whereas each flex of $C\in \mathscr{C}$ correponds to  a smooth point of $C^{\vee}$. The singularities of $\tilde{C} ^{\vee}$ consist of $\frac{1}{2}d(d-2)(d-3)(d+3)$ ordinary nodes and $3d(d-2)$ cusps.

As a special case, we consider the singularities of the dual curve of the Fermat curve $C_0\in \mathscr{C}$ of degree $q^2+q+1$. We will show that the dual curve $C_0 ^{\vee}$ is related to the Ballico-Hefez curve.

Let $\gamma _d : \mathbb{P} ^2 \to \mathbb{P} ^2 $ be a morphism defined by $\gamma _d ([x_0 : x_1 : x_2]) = [x_0 ^d : x_1 ^d : x_2 ^d ]$, and $l_0$ be a line $x_0+x_1+x_2=0 $ in $\mathbb{P} ^2 $. 
\begin{definition}
The \emph{Ballico-Hefez curve} is the image of the line $l_0$ of the morphism $\gamma _{q+1} $.
\end{definition}
In \cite{MR3323512}, Hoang and Shimada define the Ballico-Hefez curve to be the image of the morphism $\mathbb{P} ^1 \to \mathbb{P} ^2 $ defined by 
\[
[s : t] \mapsto [s^{q+1} : t^{q+1} : st^q + s^q t].
\]
Note, however, that the image of this morphism is projectively isomorphic to the image of the line $l_0$ of the morphism $\gamma _{q+1} $. 
\begin{theorem}
Let $B$ be the Ballico-Hefez curve. Let $\gamma _{q^2+q+1} : \mathbb{P} ^2 \to \mathbb{P} ^2 $ be a morphism defined by the above. If $C_0 \in \mathscr{C}$ is the Fermat curve of the degree $q^2+q+1$, then
\begin{enumerate}[\normalfont (i)]
\item the dual curve $C_0^{\vee }$ is $\gamma _{q^2+q+1} ^{-1} (B) $, and

\item the singularities of $C^{\vee}_0$ consist of $(q^2+q+1)^2(q^2-q)/2$ ordinary nodes, and $3(q^2+q+1)$ singular points with the Milnor number $q^2(q+1)$ . 
\end{enumerate}
\end{theorem}
The author is grateful to Professor Ichiro Shimada for helpful comments. Part of this work was done during the author's stay in Vietnam. He is also grateful to Professor Pho Duc Tai in Vietnam National University of Science for many helpful suggestions. Moreover, the author is grateful to the referee for pointing out the author's mistakes and helpful comments.

\section{Preliminaries}
From now, let $\Bbbk$ be an algebraically closed field of characteristic $p>0$.
\begin{proposition}
Let $C$ be a plane curve. The defining polynomial of $C$ being of the form (\ref{1}) is a property independent of the choice of homogeneous coordinates.
\end{proposition}

\begin{proof}
Under the coordinates change
\[
x_i = \sum_{l=0}^{2} t_{il} y_l \ \ (t_{il} \in k),
\]
a homogeneous polynomial $F$ of the form (\ref{1}) is transformed into
\begin{equation*}
\begin{split}
F&=\sum_{i, j, k} a_{ijk} \left( \sum_{l=0}^{2} t_{il} y_l \right) \left( \sum_{m=0}^{2} t_{im} y_m \right)^q \left( \sum_{n=0}^{2} t_{in} y_n \right)^{q^2} \\
&=\sum_{i,j,k} \sum_l \sum_m \sum_n a_{ijk} t_{jl} t_{im} ^q t_{kn} ^{q^2} y_l y_m ^q y_n ^{q^2} \\
&=\sum_{l, m, n} b_{lmn} y_l y_m ^q y_n ^{q^2}, 
\end{split}
\end{equation*}
where $b_{lmn} =\displaystyle \sum_{l, m, n} a_{ijk} t_{il} t_{jm}^q t_{kn} ^{q^2}$.
\end{proof} 

\begin{lemma}
If $a_{ijk} $ are general, then the plane curve $C$ is smooth.
\end{lemma}

\begin{proof}
The Fermat curve of degree $q^2+q+1$ is smooth. Being smooth is an open condition.
\end{proof}

\section{Proof of the first half of Theorem 1}
We define the \emph{reduced Gauss map} $\Gamma _{\mathrm{red}} : C \to (\mathbb{P} ^2)^{\vee }$ of $C\in \mathscr{C} $ by
\begin{equation*}
\begin{split}
&\Gamma _{\mathrm{red}} ([x_0 : x_1 : x_2]) \\
&= \left[\left(\frac{\partial F}{\partial x_0}(x_0, x_1, x_2) \right)^{1/q} : \left(\frac{\partial F}{\partial x_1}(x_0, x_1, x_2) \right)^{1/q} : \left(\frac{\partial F}{\partial x_2}(x_0, x_1, x_2) \right)^{1/q} \right].
\end{split}
\end{equation*}
\begin{claim}
The reduced Gauss map $ C \to C^{\vee } $ is the morphism of separable degree $1$.
\end{claim}
\begin{proof}
We will prove that the degree of the dual curve of the Fermat curve of degree $q^2+q+1$ is $d(d-1)/q$, (see Section 5), and hence the reduced Gauss map of the Fermat curve is the morphism of separable degree 1. Thus the reduced Gauss map $ C \to C^{\vee } $ is also the morphism of separable degree 1.
\end{proof}
We denote the degree of a curve $C\in\mathscr{C}$ by $d = q^2+q+1$. If $C\in \mathscr{C}$ is general, then the Gauss map $\Gamma$ is an inseparable morphism of inseparable degree $q$ by (\ref{5}). Thus the degree of $C^{\vee}$ is
\[
\frac{d(d-1)}{q} = \frac{(q^2+q+1)(q^2+q)}{q} = (q^2+q+1)(q+1).
\]
In order to prove (ii) of Theorem 1, first we prove the following:
\begin{claim}
If $C\in \mathscr{C}$ is general, then the curve $C$ has no $m$-ple tangent line for $m\geq 3$.
\end{claim}
\begin{proof}
We define a variety $\mathscr{X} _1$ by 
\begin{eqnarray*}
\mathscr{X} _1=\left\{ (Q_0, Q_1, Q_2, l)\in \mathbb{P} ^2 \times \mathbb{P} ^2 \times \mathbb{P} ^2 \times (\mathbb{P} ^2 )^{\vee } \middle|
\begin{array}{ll}
Q_0 \in l,\ Q_1 \in l,\ Q_2 \in l\\
\mathrm{and}\ Q_i\neq Q_j\ \mathrm{for}\ i\neq j
\end{array}
\right\} .
\end{eqnarray*}
Then the action of $\mathrm{PGL}_3(k)$ on $\mathscr{X}_1$ is transitive. Let $(P_0, P_1, P_2, l_0)$ be a point of $\mathscr{X} _1$ and let $[x_0 : x_1 : x_2] $ be a homogeneous coordinate system such that 
\[
P_0 = [0 : 0 : 1],\ P_1=[0 : 1 : 0],\ P_2=[0 : 1 : 1]\ \mathrm{and}\ l_0 = \{ x_0 = 0 \}.
\]
Let $C$ be a plane curve in $\mathscr{C}$. We define an algebraic subset $\mathscr{D} _1$ of $\mathscr{C}$ by
\[
\mathscr{D}_1 = \left\{ Y\in \mathscr{C}\ \middle| \begin{array}{ll}
P_0, P_1\ \mathrm{and}\ P_2\ \mathrm{are\ smooth\ points\ of}\ Y,\\ 
\mathrm{and}\  T_{P_0} Y = T_{P_1} Y = T_{P_2} Y = l_0
\end{array}
\right\}.
\]
Then $C\in \mathscr{C}$ is in $\mathscr{D} _1$ if and only if
\begin{gather*}
a_{222} = 0,\ a_{122} = 0,\ a_{111} = 0,\ a_{211} = 0,\ a_{212} + a_{221} = 0,\ a_{112} + a_{121} = 0,\\
a_{022} \neq 0,\ a_{011} \neq 0\ \mathrm{and}\ a_{011} + a_{012} + a_{021} + a_{022} \neq 0.
\end{gather*}
Therefore $\mathscr{D} _1$ is of codimension 6 in $\mathscr{C} $. Since $\dim \mathscr{X} _1 = 5$, we have
\[
\dim\mathscr{X} _1 + \dim\mathscr{D} _1 < \dim\mathscr{C}.
\]
Thus if the curve $C$ is general in $\mathscr{C}$, then $C$ does not have any $m$-ple tangent line for $m\geq 3$.
\end{proof}

Second we prove the following:
\begin{claim}
If $C \in \mathscr{C}$ is general, then $\Gamma _{\mathrm{red}} $ is an immersion at every point of $C$.
\end{claim}
\begin{proof}
Let $P_0$ be the point $[0 : 0 : 1]$, and let $l_0$ be the line $\{ x_0 = 0\}$. By linear change of coordinates, we can assume that $P_0 \in C$ and $T_{P_0} C = l_0$. Let $(x, y)$ be affine coordinates such that $[x_0 : x_1 : x_2] = [x : y : 1]$. Then up to multiple constant, the polynomial $F$ can be written as
\begin{equation*}
\begin{split}
F(x, y, 1) = f(x, y) =&\ x+a_{202} x^q + a_{212} y^q + a_{002} x^{q+1} + a_{102} x^q y + a_{012} x y^q \\
 &+ a_{112} y^{q+1} + (\mathrm{terms\ of\ degree} \geq q^2). 
\end{split}
\end{equation*}
Then we have a local parametrization $x = \phi (t),\ y=t$ of $C$ at $P_0$ such that the power series $\phi (t) $ is written as 
\[
\phi (t) = -a_{212} t^q -a_{112} t^{q+1} +a_{012} a_{212} t^{2q} + \cdots.
\]
We consider the Gauss map given by (\ref{8}).
Let $(\eta , \zeta )$ be the affine coordinates of $(\mathbb{P} ^2)^{\vee}$ with the origin $l_0 \in (\mathbb{P} ^2) ^{\vee}$ such that the point $(\eta , \zeta )$ corresponds to the line $x+ \eta y + \zeta = 0$.
Then the tangent line of $C$ at $P_t = [\phi (t) : t : 1]$ is 
\[
\frac{\partial f}{\partial x} (P_t) x + \frac{\partial f}{\partial y} (P_t) y - \frac{\partial f}{\partial x} (P_t) \phi (t) -\frac{\partial f}{\partial y} (P_t) t = 0
\]
Therefore the Gauss map locally around $P_0$ is written as
\begin{equation*}
\begin{split}
\Gamma (P_t) &= \left(\frac{f_y (P_t) }{f_x (P_t) } , -\frac{f_y (P_t) }{f_x (P_t) } t - \phi (t) \right) \\
&=\left( -\frac{d\phi}{dt} (t) , t\frac{d\phi}{dt}(t) -\phi (t) \right). 
\end{split}
\end{equation*}
Since
\[
-\frac{d\phi}{dt}(t) = a_{112} t ^q + (\mathrm{terms\ of\ degree}>q)
\]
and
\[
 t\frac{d\phi}{dt} (t) -\phi (t) = a_{212} t^q + (\mathrm{terms\ of\ degree}>q),
\]
the reduced Gauss map $\Gamma _{\mathrm{red}} $ locally around $P_0$ is 
\begin{equation}
t \mapsto (\alpha _{112} t + (\mathrm{terms\ of\ degree}>1) , \alpha _{212} t + (\mathrm{terms\ of\ degree}>1) ) \label{3},
\end{equation}
where $\alpha _{ijk} = a_{ijk} ^{1/q}$. The reduced Gauss map $\Gamma _{\mathrm{red}} $ is not smooth at the point $P_0$ if and only if $\alpha _{112} = \alpha _{212} = 0$. Since the codimension of the subset
\[
\{ C \in \mathscr{C} \ |\ \alpha _{112} = \alpha _{212} =0 \}
\]
is 2 in $\mathscr{C}$, the reduced Gauss map $\Gamma _{\mathrm{red}} $ is locally immersion at every point of a general member $C$ of $\mathscr{C}$.
\end{proof}

Suppose that $C\in\mathscr{C}$ is general. We prove that the singular points of the dual curve $C^{\vee } $ are only ordinaly nodes. Let $P_0$ and $P_1$ be the points in the proof of claim 1, and let $l_0$ be the line $\{ x_0=0\}$. Suppose that $P_0$ and $P_1$ are smooth points of $C$ and $T_{P_0} C =T_{P_1} C = l_0$. Let $(x', y')$ be affine coordinates such that $[x_0 : x_1 : x_2] = [x' : 1 : y']$. Similar to the proof of the claim 2, up to multiple constant, the polynomial $F$ can be written as
\begin{equation*}
\begin{split}
F(x', 1, y') = g(x', y') =&\ x'+a_{101} x'^q + a_{121} y'^q + a_{001} x'^{q+1} + a_{201} x'^q y' + a_{021} x' y'^q  \\
 &+ a_{221} y'^{q+1} + (\mathrm{terms\ of\ degree} \geq q^2 ). 
\end{split}
\end{equation*}
Then we have a local parametrization $x' = \psi (t),\ y'=t$, of $C$ at $P_0$ such that the power series $\psi (t) $ is written as 
\[
\psi (t) = -a_{121} t^q -a_{221} t^{q+1} +a_{021} a_{121} t^{2q} + \cdots.
\]
Let $(\eta, \zeta )$ be the affine coordinates of $(\mathbb{P} ^2)^{\vee}$ with the origin $l_0 \in (\mathbb{P} ^2) ^{\vee}$ such that the point $(\eta , \zeta )$ corresponds to the line $x'+ \eta y' + \zeta = 0$. The tangent line of $C$ at $P'_t = [\psi (t) : 1 : t] $ is
\[
\frac{\partial g}{\partial x'} (P'_t) x' + \frac{\partial g}{\partial y'} (P'_t) y' - \frac{\partial g}{\partial x'} (P'_t) \phi (t) -\frac{\partial g}{\partial y'} (P'_t) t = 0.
\]
Therefore the Gauss map $\Gamma $ locally around $P_1$ is written as
\begin{equation*}
\begin{split}
\Gamma (P_t') &= \left(\frac{g_{y'} (P'_t) }{g_{x'} (P'_t) } , -\frac{g_{y'} (P'_t) }{g_{x'} (P'_t) } t - \psi (t) \right) \\
&=\left( -\frac{d\psi}{dt}(t) , t\frac{d\psi}{dt} (t) -\psi (t) \right). 
\end{split}
\end{equation*}
Since
\[
-\frac{d\psi}{dt} (t) = a_{221} t ^q + (\mathrm{terms\ of\ degree}>q)
\]
and
\[
 t\frac{d\psi}{dt} (t) -\psi (t) = a_{121} t^q + (\mathrm{terms\ of\ degree}>q),
\]
we describe the reduced Gauss map
\begin{equation}
t \mapsto (\alpha _{221} t + (\mathrm{terms\ of\ degree}>1) ,\ \alpha _{121} t + (\mathrm{terms\ of\ degree}>1) ) \label{4}
\end{equation}
 locally around $P_1$. We define a variety $\mathscr{X}_2$ by 
\[
\mathscr{X}_2 = \{ (Q_0, Q_1, l) \in \mathbb{P}^2 \times \mathbb{P}^2 \times (\mathbb{P}^2)^{\vee} \ |\ Q_0 \in l, Q_1 \in l\ \mathrm{and}\ Q_0 \neq Q_1\}.
\]
Then the action of $\mathrm{PGL}_3(k)$ on $\mathscr{X}_2$ is transitive and $\dim\mathscr{X}_2 = 4$. Let $(P_0, P_1, l_0)$ be the point of $\mathscr{X}_2$ such that $P_0 = [0 : 0 : 1]$, $P_1 = [0 : 1 : 0]$ and $l_0 = \{x_0=0\}$. We define a subset $ \mathscr{D}_2$ of $\mathscr{C}$ by 
\[
\mathscr{D} _2 = \{ Y\in \mathscr{C}\ |\ P_0\ \mathrm{and}\ P_1\ \mathrm{are\ smooth\ points\ of}\ Y,\ \mathrm{and}\ T_{P_0} Y = T_{P_1} Y = l_0\}
\]
Then $C\in \mathscr{D}_2$ if and only if
\[
a_{222} = 0,\ a_{122} = 0,\ a_{111} = 0,\ a_{211} = 0,\ a_{022} \neq 0,\ a_{011} \neq 0.
\]
Thus the codimension of $\mathscr{D}_2$ is 4. For $C\in \mathscr{D}_2$,  by (\ref{3}) and (\ref{4}), the singularities of $C^{\vee}$ at the point $l_0$ is not a ordinary node if and only if
\[
\begin{vmatrix}
\alpha_{112} & \alpha_{212} \\
\alpha_{211} & \alpha_{121} \\
\end{vmatrix}
=0.
\]
We define a subset $\mathscr{D}_2'$ of $\mathscr{C}$ by
\[
\mathscr{D}_2' = \left\{ Y\in \mathscr{C}\ \middle| \begin{array}{ll}
P_0\ \mathrm{and}\ P_1\ \mathrm{are\ smooth\ points\ of}\ Y,\\
T_{P_0} Y = T_{P_1} Y = l_0,\ \mathrm{and}\ Y^{\vee}\ \mathrm{does\ not\ have\ ordinary\ node\ at}\ l_0
\end{array}
\right\}.
\]
Since the codimension of $\mathscr{D}_2'$ is 5, 
\[
\dim\mathscr{D}_2' + \dim\mathscr{X}_2 < \dim\mathscr{C}.
\]
Therefore, since $a_{ijk} $ are general, the dual curve $C^{\vee }$ has only ordinary nodes as its singularities.

\section{Proof of the second half of Theorem 1}
\subsection{Number of the ordinary nodes of $C^{\vee}$}
Let $g$ and $g^{\vee }$ be the genera of a general curve $C\in \mathscr{C}$ and its dual curve $C^{\vee}$, respectively. Let $\delta$ be the number of the ordinary nodes of $C^{\vee}$. Then
\[
g=\frac{(d-1)(d-2)}{2} = \frac{\{ (q^2+q+1)-1\}\{ (q^2+q+1)-2\} }{2} 
\]
and
\begin{equation*}
\begin{split}
g^{\vee} =& \frac{(d^{\vee}-1)(d^{\vee}-2)}{2} - \delta \\
=& \frac{\{(q^2+q+1)(q+1)-1\} \{(q^2+q+1)(q+1)-2\}}{2} - \delta ,
\end{split}
\end{equation*}
where $d$ and $d^{\vee}$ are the degree of $C$ and $C^{\vee}$, respectively, because, by the previous section, $C^{\vee}$ has only ordinary nodes. By claim 2 of section 3, the reduced Gauss map $\Gamma _{\mathrm{red}} $ is birational onto its image. Thus $g = g^{\vee}$ and hence we have
\begin{equation*}
\begin{split}
\delta =& \frac{\{(q^2+q+1)(q+1)-1\} \{(q^2+q+1)(q+1)-2\}}{2} \\
&- \frac{\{ (q^2+q+1)-1\}\{ (q^2+q+1)-2\} }{2} \\
=& \frac{q(q^2+q+1)(q^3+3q^2+3q-1)}{2}
\end{split}
\end{equation*}

\subsection{Number of the flexes}
We denote by $\mathrm{mult}_{P} (D_1, D_2)$ the intersection multiplicity of projective plane curves $D_1$ and $D_2$ at a point $P \in D_1 \cap D_2$.
\begin{lemma}\label{y}
We suppose that $C \in \mathscr{C}$ is a general plane curve in $\mathscr{C}$. If the multiplicity $\mathrm{mult}_u (T_u C, C)$ is more than $q$ at $u\in C$, then the multiplicity $\mathrm{mult}_u (T_u C, C)$ is $q+1$ at $u\in C$ and all other intersection points of $T_uC$ and C are not tangent point.
\end{lemma}
\begin{proof}
We use the same notation as in Section 3. We define a variety $\mathscr{X}_0$ by
\[
\mathscr{X}_0 = \{ (Q, l) \in \mathbb{P} ^2 \times (\mathbb{P} ^2 ) ^{\vee }\ |\ Q\in l \} .
\]
Then the action of $\mathrm{PGL}_3(k)$ on $\mathscr{X}_0$ is transitive and $\mathrm{dim} \mathscr{X}_0 = 3$. We recall that $[x_0 : x_1 : x_2 ]$ are homogeneous coordinates, $P_0 = [ 0 : 0 : 1]$, $P_1 = [ 0 : 1 : 0 ]$ and $l_0 = \{ x_0 = 0\} $. We define two subsets $\mathscr{D}_0 $ and $\widetilde{\mathscr{D} } _0$ of $\mathscr{C}$ by
\[
\mathscr{D} _0 = \left\{ Y\in \mathscr{C} \middle|\ \begin{array}{ll}
P_0\ \mathrm{is\ the\ smooth\ point\ of}\ Y,\ T_{P_0} Y = l_0\\
\mathrm{and}\ \mathrm{mult}_{P_0} (T_{P_0} Y, Y) = q+1
\end{array}
\right\}
\]
and
\[
\widetilde{\mathscr{D} } _0 = \left\{ Y\in \mathscr{C} \middle|\ \begin{array}{ll}
P_0\ \mathrm{is\ the\ smooth\ point\ of}\ Y,\ T_{P_0} Y = l_0\\
\mathrm{and}\ \mathrm{mult}_{P_0} (T_{P_0} Y, Y) > q+1
\end{array}
\right\} .
\]
Then the curve $C \in \mathscr{D}_0$ if and only if
\[
a_{222} = 0,\ a_{122} = 0,\ a_{212} = 0,\ a_{112} \neq 0\ \mathrm{and}\ a_{022}\neq 0,
\]
and $C\in \widetilde{\mathscr{D} } _0$ if and only if
\[
a_{222} = 0,\ a_{122} = 0,\ a_{212} = 0,\ a_{112} = 0\ \mathrm{and}\ a_{022}\neq 0.
\]
Therefore the codimension of $\mathscr{D} _0$ is 3 and that of $\widetilde{\mathscr{D} } _0$ is more than 3 in $\mathscr{C}$. Thus we have
\[
\mathrm{dim} \mathscr{X}_0 + \mathrm{dim} \widetilde{\mathscr{D} } _0 < \mathrm{dim} \mathscr{C}.
\]
We proved the first half of the lemma. We define a subset $\widetilde{\mathscr{D} } _2$ of $\mathscr{C}$ by
\[
\widetilde{\mathscr{D} } _2 = \left\{ Y\in \mathscr{C} \middle|\ \begin{array}{ll}
P_0\ \mathrm{and}\ P_1\ \mathrm{are\ the\ smooth\ points\ of}\ Y,\ T_{P_0} Y = l_0,\\
T_{P_1} Y= l_0\ \mathrm{and}\ \mathrm{mult}_{P_0} (T_{P_0} Y, Y) = q+1
\end{array}
\right\} .
\]
Then the curve $C\in \widetilde{\mathscr{D} } _2$ if and only if
\[
a_{222} = 0,\ a_{122} = 0,\ a_{111} = 0,\ a_{211} = 0,\ a_{212} = 0,\ a_{112} \neq 0,\ a_{022} \neq 0,\ a_{011} \neq 0.
\]
Therefore codimension of $\widetilde{\mathscr{D} } _2$ is 5, and we recall $\mathrm{dim} \mathscr{X}_2 = 4$. Thus, since we have
\[
\mathrm{dim} \mathscr{X}_2 + \mathrm{dim} \widetilde{\mathscr{D} } _2 < \mathrm{dim} \mathscr{C},
\]
the second half of the lemma is proved.
\end{proof}
Let $g$ be the genus of a general curve $C\in \mathscr{C} $. 
We use the notion and notation about the correspondence of curves introduced in \cite[Chap. 2, Section 5]{MR1288523}. Let $T : C \to C$ be correspondence defined by $T(u) = T_{u} C.C-qu $, $D\subset C\times C$ its curve of correspondence, i.e. $D =\overline{ \{ (u,v)\ |\ u\neq v,\ v \in T_u C \} }$. 
Then the degree of $T$ is
\[
\deg T = (q^2+q+1)-q = q^2+1.
\]
Let $\pi _2 : C \times C \to C$ be the projection on second factor. In order to find the degree of $T^{-1}$, we have to caluculate the number of tangent lines to $C$, (counted with the intersection multiplicities of $D$ and $\pi_2 ^{-1} (v)$) other than $T_v C$ passing throght a general point $v\in C$. We consider the projection $\pi _v : C\to \mathbb{P}^1$ from the center $v\in C$ onto a line. Let $\Omega _{C / \mathbb{P}^1 }$ be the sheaf of the relative differential of $C$ over $\mathbb{P} ^1$. By Hurwitz-formula \cite[Chap. I\hspace{-.1em}V, Corollary 2.4]{MR0463157},
\[
2g-2 = -2(q^2+q) + \deg R ,
\]
where the divisor $R$ is the ramification divisor of $\pi _v$ i.e. $R = \sum_{u\in C} \mathrm{length} (\Omega _{C / \mathbb{P}^1 })_u u$. Hence
\[
\deg R = q^4+2q^3+2q^2+q-2.
\]
Moreover, the length of $(\Omega _{C / \mathbb{P}^1 })_v$ is $q-2$. Hence, we have
\begin{equation*}
\begin{split}
\deg T^{-1} &= (q^4+2q^3+2q^2+q-2)-(q-2) \\
&= q^4+2q^3+2q^2.
\end{split}
\end{equation*}

\begin{lemma}
Let $\pi _1,\ \pi _2 : C\times C\to C$ be the projections on first and second factors, respectively. The divisor $D$ on $C\times C$ is algebraically equivalent to 
\[
(q^4+2q^3+2q^2+q)E_u +(q^2+q+1)F_v -q\Delta, 
\]
where $E_u = \pi _1 ^{-1} (u),\ F_v = \pi _2 ^{-1} (v)$ and $\Delta \subset C\times C$ is the diagonal.
\end{lemma}

\begin{proof}
For some $u_0,\ v_0\in C$, we write 
\[
T(u_0) + qu_0 = \sum b_i v_i
\]
and
\[
T^{-1} (v_0)+qv_0 = \sum a_iu_i.
\]
Let $L$ be the line bundle
\[
L = D-\sum a_i E_{u_i} - \sum b_i F_{v_i} + q\Delta.
\]
For any $x\in C$, the restriction of $L$ to $E_x$ is trivial because the divisor $T(x) + qx$ is linearly equivalent to $T(u_0) + qu_0$. By \cite[Chap. III, Exercise 12.4]{MR0463157}, there is a line bundle $M$ on $C$ such that $L \cong \pi_1 ^{\ast} (M)$. Since the restriction of $L$ to $F_{v_0}$ is trivial, the line bundle $L$ is also trivial. Thus $D$ is linearly equivalent to 
\[
\sum a_i E_{u_i} + \sum b_i F_{v_i} - q\Delta.
\]
For any  $u,\ v\in C$, the divisors $E_{u_i}$ (resp. $F_{v_i} $) are algebraically equivalent to $E_u$ (resp. $F_v$). Note that the degrees of $T(u_0) + qu_0$ and $T^{-1} (v_0) + qv_0$ are
\[
\deg (T(u_0) + qu_0 ) = q^2+q+1
\]
and
\[
\deg (T^{-1} (v_0)+qv_0 ) = q^4+2q^3+2q^2+q, 
\]
and hence the result is proved.
\end{proof}
\begin{lemma}\label{w}
If $C\in\mathscr{C}$ is a general plane curve in $\mathscr{C}$, then $D$ and $\Delta$ intersect transversally at any point $(u, v) \in D \cap \Delta $.
\end{lemma}
\begin{proof}
We use the same notations as in Section 3 and Lemma \ref{y}. We recall that $[x_0 : x_1 : x_2]$ is homogeneous coordinates, $P_0 = [0 : 0 : 1]$, $l_0 = \{ x_0 = 0 \}$ and
\[
\mathscr{D} _0 = \left\{ Y\in \mathscr{C} \middle|\ \begin{array}{ll}
P_0\ \mathrm{is\ the\ smooth\ point\ of}\ Y,\ T_{P_0} Y = l_0\\
\mathrm{and}\ \mathrm{mult}_{P_0} (T_{P_0} Y, Y) = q+1
\end{array}
\right\}.
\]
By change of coordinates, we assume that $C \in \mathscr{D}_0$. Let $(x, y)$ be affine coordinates such that $[x_0 : x_1 : x_2] = [x : y : 1]$. Then up to multiple constant, the polynomial $F$ can be written as
\begin{equation*}
\begin{split}
F(x, y, 1) =\ & x + a_{202} x^q + a_{002} x^{q+1} + a_{102} x^q y + a_{012} xy^q + a_{112} y^{q+1} \\ 
& + a_{220}x^{q^2} + a_{221} y^{q^2} + (\mathrm{terms\ of\ degree} > q^2).
\end{split}
\end{equation*}
Then we have a local parametrization $x=\phi_1 (t)$, $y = t$ of $C$ at $P_0$ such that the power series $\phi_1 (t)$ is written as
\[
\phi_1 (t) = -a_{112} t^{q+1} +a_{012} a_{112}t^{2q+1} + \cdots - a_{221} t^{q^2} + (\mathrm{terms\ of\ degree} > q^2).
\]
Let $(P_{t_1} , P_{t_2})$ be a point of $D$ in a small neighborhood of $(P_0, P_0)$ such that 
\[P_{t_1} = [\phi_1 (t_1) : t_1 : 1]\ \mathrm{and}\ P_{t_2} = [\phi_1 (t_2) : t_2 : 1].
\]
The tangent line of $C$ at $P_{t_1}$ is 
\[
x = \frac{d\phi_1}{dt} (t_1)y-t_1\frac{d\phi_1}{dt} (t_1) + \phi_1(t_1),
\]
and hence
\begin{equation*}
\begin{split}
x  = &\ (-a_{112} t_1^q +a_{012}a_{112}t_1^{2q} + (\mathrm{terms\ of\ degree} >2q ))y \\
 & + (-a_{221}t_1^{q^2} + (\mathrm{terms\ of\ degree}>q^2 )).
\end{split}
\end{equation*}
Therefore $t_2$ is the solution of the equation
\begin{equation}\label{z}
\frac{d\phi_1}{dt} (t_1)y-t_1\frac{d\phi_1}{dt} (t_1) + \phi_1(t_1) - \phi_1(y) = 0
\end{equation}
for $y$ that is not $t_1$ and approaches to 0 when $t_1$ tends to 0. We can express the left hand side of (\ref{z}) as
\begin{equation*}
\begin{split}
& (-a_{112} t_1^q +a_{012}a_{112}t_1^{2q} + (\mathrm{terms\ of\ degree} >2q\ \mathrm{in}\ t_1))y \\
&+ (-a_{221}t_1^{q^2} + (\mathrm{terms\ of\ degree}>q^2\ \mathrm{in}\ t_1)) \\
&+ a_{112} y^{q+1} -a_{012} a_{112}y^{2q+1} + \cdots + a_{221} y^{q^2} + (\mathrm{terms\ of\ degree} > q^2\ \mathrm{in}\ y) \\
= &\ (y-t_1)^qf_{t_1}(y),
\end{split}
\end{equation*}
where the power series $f_{t_1}(y)$ is written as
\[
f_{t_1}(y) = a_{112}y + a_{221} t_1^q + a_{221} y^q + \cdots.
\]
Since $C\in\mathscr{D}_0$, $a_{112} \neq 0$. Thus a solution of $f_{t_1}(y) = 0$ is
\[
y = -\frac{a_{221}}{a_{112}} t_1^q + (\mathrm{terms\ of\ degree} > q).
\]
Therefore we have 
\[
t_2 = -\frac{a_{221}}{a_{112}} t_1^q + (\mathrm{terms\ of\ degree} > q).
\]
If $(P_{t_1} , P_{t_2})$ is a point in $\Delta $, then $t_1 = t_2$. Therefore, if $(P_{t_1} , P_{t_2}) \in D \cap \Delta$, then 
\[
t_1 = -\frac{a_{221}}{a_{112}} t_1^q + (\mathrm{terms\ of\ degree} > q).
\]
Thus $D$ and $\Delta$ intersect transversally at $(P_0, P_0) \in D \cap \Delta $.

\end{proof}
By Lemma \ref{w}, the number of the flexes is equal to the intersection number $(D\cdot \Delta)$ for a general member $C$ of $\mathscr{C}$. Since the self-intersection number of $\Delta$ is $2-2g$, the intersection number $(D\cdot \Delta )$ is
\begin{equation*}
\begin{split}
(D\cdot \Delta )  &= (\{ (q^4+2q^3+2q^2+q)E_u +(q^2+q+1)F_v -q\Delta \} \cdot \Delta ) \\
&= q^4+2q^3+3q^2+2q+1-q(2-2g ) \\
&= q^5+3q^4+2q^3+2q^2+1\\
&= (q^3+2q^2-q+1)(q^2+q+1).
\end{split}
\end{equation*}
\section{Fermat curve}
For any formal power series $f \in \Bbbk [[x, y]]$, we define the {\em Milnor number} $\mu (f)$ by
\[
\mu (f) = \dim _{\Bbbk } \Bbbk [[x, y]] / \left(\frac{\partial f}{\partial x}, \frac{\partial f}{\partial y} \right).
\]
Calculation method of the Milnor number for a formal power series in characteristic zero is well known. (For example, see \cite{MR2107253}.) However, in positive characteristic, the calculation method and result of the Milnor number differ from the characteristic-zero case in general. In the case of the following lemma, however, the Milnor number is the same as the characteristic-zero case.

\begin{lemma}
Let $a$ and $b$ be elements in $\Bbbk\setminus \{ 0\}$, and let $f\in \Bbbk[[x,y]]$ be a formal power series defined by 
\[
f(x, y) = ax^{\alpha } + by^{\beta } + \sum_{\alpha \beta < \alpha s + \beta r} c_{r, s} x^r  y^s,
\]
where $\alpha $ and $\beta $ satisfy $p\mathrel{\not | } \alpha$, $p\mathrel{\not | } \beta$ and are relatively prime. Then the Milnor number $\mu (f) $ of $f$ is
\[
\mu (f) = (\alpha -1)(\beta -1).
\]
\end{lemma}
\begin{proof}
We use notations of \cite{MR3839793}. The $(\beta, \alpha)$-order of $f$ is
\[
\mathrm{ord} _{(\beta, \alpha)} (f) = \alpha \beta.
\]
The $(\beta, \alpha)$-initial of $f$ is 
\[
\mathrm{in} _{(\beta, \alpha)} (f) =ax^{\alpha} + by^{\beta}.
\]
Thus the formal power series $f$ is the semi-quasihomogeneous with respect to $(\beta, \alpha)$. By the Appendix of \cite{MR3839793}, 
\[
\mu (f) = (\alpha -1)(\beta -1).
\]
\end{proof}

\begin{proof}[Proof of Theorem 2]
The morphisms $\gamma _{q^2+q+1}$ and $\gamma _{q+1} $ satisfy
\[
\gamma _{q^2+q+1} \circ \gamma _{q+1} = \gamma _{q+1} \circ \gamma _{q^2+q+1} = \gamma _{(q^2+q+1)(q+1)}.
\]
By the definition of the Ballico-Hefez curve and the line $l = \gamma _{q^2+q+1} (C_0)$, we have
\[
B=\gamma _{q+1} (l) = \gamma _{q+1} (\gamma _{q^2+q+1} (C_0) ) = \gamma _{q^2+q+1} (\gamma _{q+1} (C_0) ) = \gamma _{q^2+q+1} (C^{\vee } _0),
\]
and hence (i) is proved.

We define $X\subset \mathbb{P} ^2$ by 
\[
X= \{x_0=0 \} \cup \{x_1=0 \} \cup \{x_2=0 \} .
\]
The Ballico-Hefez curve $B$ has $\frac{q^2-q}{2} $ ordinary nodes on $\mathbb{P} ^2\setminus X$ (see \cite[Theorem 2.2]{MR2961398}), and no singular points on $X$.
Let $H$ and $h$ be the defining polynomials of $C^{\vee } _0$ and $B$, respectively. Using Proposition 1.6 of \cite{MR3323512}, if $p=2$, then
\begin{equation}
\begin{split}
h &=  x_0 ^{q+1} + x_1^{q+1} + x_2^{q+1} +x_0^q x_2+x_1^q x_2 + x_0 x_2^q + x_1 x_2^q \\
&\quad + \sum_{i=0}^{\nu -1} x_0^{2^i} x_1^{2^i} (x_0 + x_1+ x_2) ^{q+1-2^{i+1} } , \label{6}
\end{split}
\end{equation}
whereas if $p$ is odd, then
\begin{equation}
\begin{split}
h &= x_0^{q+1} + x_1^{q+1} + x_2^{q+1} -x_0^q x_1 - x_0^q x_2 -x_0 x_1^q - x_1^q x_2 -x_0 x_2^q -x_1 x_2^q \\
 &\quad +  (x_0^2 + x_1^2 + x_2 ^2 - 2x_0x_1 - 2x_1x_2 - 2x_2x_0)^{\frac{q+1}{2} } . \label{7}
\end{split}
\end{equation}
By (i), the polynomial $H$ satisfies $H(x_0 , x_1, x_2 ) = h(x_0^{q^2+q+1} , x_1^{q^2+q+1}, x_2^{q^2+q+1} ) $, and two polynomials $H$ and $h$ are symmetric under the permutation of coordinates $x_0$, $x_1$ and $x_2$. First we consider the singularities of $C^{\vee } _0$ on $\mathbb{P} ^2 \setminus X$. The morphism $\gamma _{q^2+q+1} : \mathbb{P} ^2 \setminus X \to \mathbb{P} ^2 \setminus X $ is \'{e}tale of degree $(q^2+q+1)^2$. Thus, the ordinary nodes of $C^{\vee} _0$ on $\mathbb{P} ^2 \setminus X$ are $(q^2+q+1)^2(q^2-q)/2 $. 

Next, we consider the singularities of $C^{\vee } _0$ on $X$. $h(0, x_1, x_2) = 0$ if and only if $x_1 = x_2$ by (\ref{6}) and (\ref{7}). Moreover, the polynomial $H$ and its partial derivatives  $\partial H/\partial x_i  = x_i^{q^2+q} (\partial h/\partial x_i )$ vanish at a point in $\{ x_0 = 0 \} $. Thus all the points on $C_0^{\vee} \cap \{ x_0 = 0 \} $ are singular points of $C_0^{\vee}$. The morphism $\gamma _{q^2+q+1} |_{\{ x_0=0\} } $ restricted to $\{x_0=0\} $ is degree $q^2+q+1$. Thus the number of the singular points of $C^{\vee } _0$ on $\{ x_0=0 \} $ are $q^2+q+1$. Therefore, by the polynomial $H$ is symmetric, the number of the singular points of $C^{\vee } _0$ on $X$ are $3(q^2+q+1)$. 

Finally, since all Milnor numbers at points in $\gamma_{q^2+q+1} ^{-1} ([ 0 : 1 : 1 ]) $ are equal, we should caluculate the Milnor number at the point $[0 : 1 : 1 ] \in C^{\vee } _0$. 
If $p=2$, 
\begin{equation*}
\begin{split}
h(x_0^{q^2+q+1}, x_1+1, 1) &= x_0^{q^2+q+1} + x_1^{q+1} + x_0^{q(q^2+q+1)} +x_0^{(q+1)(q^2+q+1)} \\
 &\quad + \sum_{i=0}^{\nu -1} (x_0^{q^2+q+1} )^{2^i} (x_1+1)^{2^i} (x_0^{q^2+q+1} + x_1 )^{q+1-2^i },
\end{split}
\end{equation*}
whereas if $p$ is odd,
\begin{equation*}
\begin{split}
h(x_0^{q^2+q+1}, x_1+1, 1) &= -2x_0^{q^2+q+1} + x_1^{q+1} + x_0^{(q^2+q+1)(q+1)} \\ 
&\quad - x_0^{q(q^2+q+1)} x_1 - 2x_0 ^{q(q^2+q+1)} - x_0^{q^2+q+1} x_1^q \\
&\quad + (x_0^{2(q^2+q+1)} + x_1^2 -2x_0^{q^2+q+1} x_1 - 4x_0^{q^2+q+1} )^{\frac{q+1}{2} }.
\end{split}
\end{equation*}
By Lemma 3.1, the Milnor number of $h(x_0^{q^2+q+1}, x_1+1, 1)$ is 
\[
q(q^2+q)=q^2(q+1).
\]
\end{proof}
We confirm that the genus of the Fermat curve agree with the genus of its dual curve. The genus $g$ of the Fermat curve $C_0$ of the degree $d=q^2+q+1$ is 
\[
g=\frac{(d-1)(d-2)}{2} = \frac{(q^2+q)(q^2+q-1)}{2}. 
\]
Let $\mu _P$ be the Milnor number and let $r_P$ be the number of the branches at a singular point of the dual curve $C_0^{\vee}$. If a point $P\in C_0^{\vee}$ is an ordinary node, then $\mu_P = 1$ and $r_P =2$, whereas if a point $P$ is in $C_0^{\vee} \cap X$, then $\mu_P = q^2(q+1)$ and $r_P =1$. Thus the degree $d^{\vee}$ of $C_0^{\vee}$ is $ (q+1)(q^2+q+1)$, and the genus $g^{\vee}$ of $C_0^{\vee}$ is
\begin{equation*}
\begin{split}
g^{\vee} &= \frac{(d^{\vee} -1)(d^{\vee} -2)}{2} - \frac{1}{2} \sum_{P\in \mathrm{Sing} C_0^{\vee} } (\mu_P + r_P -1) \\
&= \frac{\{(q^2+q+1)(q+1)-1\} \{(q^2+q+1)(q+1)-2\}}{2} \\
&\quad -\frac{1}{2} \{ (q^2+q+1)^2(q^2-q)+3(q^2+q+1)q^2(q+1)\} \\
&= \frac{(q^2+q)(q^2+q-1)}{2}.
\end{split}
\end{equation*}

\bibliographystyle{abbrv}

\begin{thebibliography}{1}

\bibitem{MR2961398}
{\sc S.~Fukasawa, M.~Homma, and S.~J. Kim},
Rational curves with many rational points over a finite field,
Arithmetic, geometry, cryptography and coding theory, Contemp. Math. {\bf 574}, Amer. Math. Soc., Providence, RI, 2012, 37--48.

\bibitem{MR3839793}
{\sc E.~R. Garc\'{\i}a~Barroso and A.~P\l oski},
On the {M}ilnor formula in arbitrary characteristic,
Singularities, algebraic geometry, commutative algebra, and related topics, Springer, Cham, 2018, 119--133.

\bibitem{MR1288523}
{\sc P.~Griffiths and J.~Harris},
Principles of algebraic geometry,
Wiley Classics Library, John Wiley \& Sons, Inc., New York, 1994.

\bibitem{MR0463157}
{\sc R.~Hartshorne},
Algebraic geometry,
Graduate Texts in Math. {\bf 52},
Springer-Verlag, New York, 1977.

\bibitem{MR3323512}
{\sc T.~H. Hoang and I.~Shimada},
On {B}allico-{H}efez curves and associated supersingular surfaces,
Kodai Math. J. {\bf 38} (2015), 23--36.

\bibitem{MR2107253}
{\sc C.~T.~C. Wall},
Singular points of plane curves, 
London Mathematical Society Student Texts {\bf 63},
Cambridge University Press, Cambridge, 2004.

\end{thebibliography}

\end{document}